\documentclass[12pt,a4paper]{amsart}
\usepackage{amsfonts}
\usepackage{amsthm}
\usepackage{amsmath, amscd, amssymb}
\usepackage[latin2]{inputenc}
\usepackage{t1enc}
\usepackage[mathscr]{eucal}
\usepackage{indentfirst}
\usepackage{graphicx}
\usepackage{graphics}
\usepackage{hyperref}
\usepackage{pict2e}
\usepackage{epic}
\numberwithin{equation}{section}
\usepackage[margin=2.9cm]{geometry}
\usepackage{epstopdf}

\DeclareMathOperator{\reg}{reg}

\DeclareMathOperator{\lex}{lex}
\DeclareMathOperator{\pol}{pol}

\theoremstyle{plain}
\newtheorem{theorem}{Theorem}[section]
\newtheorem{lemma}[theorem]{Lemma}
\newtheorem{cor}[theorem]{Corollary}
\newtheorem{proposition}[theorem]{Proposition}

\theoremstyle{definition}
\newtheorem{definition}[theorem]{Definition}
\newtheorem{remark}[theorem]{Remark}

\begin{document}

\title{Powers of ideals associated to $(C_4, 2K_2)$-free graphs}

\author[N. Erey]{Nursel Erey}

\address{Gebze Technical University \\  Department of Mathematics, Gebze, Kocaeli, 41400, Turkey } 

\email{nurselerey@gtu.edu.tr}



\begin{abstract} 
Let $G$ be a $(C_4, 2K_2)$-free graph with edge ideal $I(G)\subset \Bbbk[x_1,\dots , x_n]$. We show that $I(G)^s$ has linear resolution for every $s\geq 2$. Also, we show that every power of the vertex cover ideal of $G$ has linear quotients. As a result, we describe the Castelnuovo-Mumford regularity of powers of $I(G)^{\vee}$ in terms of the maximum degree of $G$.
\end{abstract}

\maketitle

\section{Introduction}

Let $G$ be a simple graph with the vertex set $V=\{x_1,\dots, x_n\}$ and let $S =\Bbbk[x_1,\dots, x_n]$ be the polynomial ring over a field $\Bbbk$.  The \emph{edge ideal} $I(G)$ of $G$ is the ideal generated by the monomials $x_ix_j$ where $\{x_i,x_j\}$ is an edge of $G$. A set of vertices $C$ is called a \emph{vertex cover} if every edge of $G$ contains a vertex of $C$. The \emph{vertex cover ideal} of $G$, denoted by $I(G)^{\vee}$, is generated by the squarefree monomials $x_{i_1}\dots x_{i_k}$ where $\{x_{i_1}, \dots , x_{i_k}\}$ is a vertex cover of $G$. The vertex cover ideal of $G$ is also known as \emph{Alexander dual} of $I(G)$. 

Minimal free resolutions of edge ideals and vertex cover ideals are extensively studied and recently some attention has been given to regularity of powers of these ideals. It is well-known \cite{cutkosky, kodiyalam} that for any graded ideal $I$ there exist non-negative integers $c, d, s_0$ such that $\reg(I^s)=ds+c$ for all $s\geq s_0$. Although the constant $d$ can be determined from the generators of the ideal, no explicit formulas are known for $c$ and $s_0$. When $I$ is an edge or cover ideal, the integers $c$ and $s_0$ were computed for some families of graphs \cite{ali, banerjee, selvi, erey, hang trung, herzog hibi zheng, jayanthan bipartite, jayanthan, moghimian et al}.

By Fr\"oberg's Theorem~\cite{froberg} the edge ideal $I(G)$ has a linear resolution if and only if $G$ is co-chordal, i.e., the complement graph $G^c$ is chordal. Herzog, Hibi and Zheng \cite{herzog hibi zheng} showed that if $G$ is a co-chordal graph, then all powers of $I(G)$ have linear resolutions. Recall that a graph is \emph{chordal} if it has no induced cycle of length $4$ or more. Francisco, H\`{a} and Van Tuyl showed that if some power of $I(G)$ has linear resolution, then the complement graph $G^c$ has no induced cycle of length $4$, i.e., $G$ is \emph{gap-free}. In this direction Peeva and Nevo \cite{peeva nevo} asked the following question regarding this wider family of graphs: If $G$ is a gap-free graph, then does $I(G)^s$ have a linear resolution for every $s\gg 0$? 

In Section \ref{sec:edge ideals} we answer the question above in the affirmative under the additional assumption that $G^c$ is gap-free. If both $G$ and $G^c$ are gap-free, then $G$ is also known as $(C_4, 2K_2)$-free. Our proof is based on a structural characterization of such graphs given in \cite{blazsik} and a method given in \cite{banerjee} to bound the regularity of powers of edge ideals.

A graded ideal $I$ is called \emph{componentwise linear} if for each $d$, the ideal generated by all degree $d$ elements of $I$ has a linear resolution. The notion of componentwise linear ideal generalizes that of ideal with linear resolution. For a graded ideal, a stronger property than being componentwise linear is having linear quotients:
\[ \begin{array}{rcc}
\text{linear quotients} & \implies & \text{componentwise linear} \\
 &  & \Uparrow \\
 \text{equigenerated ideal with linear quotients} & \implies & \text{linear resolution}
 \end{array}\] 

A graph $G$ is called (\emph{sequentially}) \emph{Cohen-Macaulay} if the quotient ring $S/I(G)$ is (sequentially) Cohen-Macaulay over every field $\Bbbk$. Due to a result of Eagon and Reiner \cite{eagon reiner} it is known that a graph is Cohen-Macaulay if and only if its vertex cover ideal has a linear resolution. More generally, it was proved in \cite{herzog hibi componentwise} that a graph is sequentially Cohen-Macaulay if and only if its vertex cover ideal is componentwise linear. 

In general, powers of an ideal with linear resolution may not have linear resolutions. Sturmfels~\cite{sturmfels} gave an example of an ideal which has linear quotients but the second power of the ideal has no linear resolution. In particular, the ideal $I= (def, cef, cdf, cde, bef, bcd, acf, ade)$ has linear quotients but $\reg(I^2)=7$ while $I^2$ is generated in degree $6$. On the other hand, one may ask the following question: Given a (sequentially) Cohen-Macaulay graph, what can be said about the powers of its vertex cover ideal? Francisco and Van Tuyl \cite{francisco van tuyl} proved that the vertex cover ideal of a chordal graph is componentwise linear. Herzog, Hibi and Ohsugi \cite{herzog hibi ohsugi} showed that powers of the vertex cover ideal of a Cohen-Macaulay chordal graph have linear resolutions. Furthermore, they conjectured that all powers of the vertex cover ideal of a chordal graph are componentwise linear. It is also known that all powers of vertex cover ideals of Cohen-Macaulay families of bipartite \cite{herzog hibi depth of powers, mohammadi moradi} and cactus \cite{mohammadi cactus} graphs have linear resolutions.

In Section \ref{sec: vertex cover ideals} we show that all powers of the vertex cover ideal of a $(C_4, 2K_2)$-free graph are componentwise linear. Our proof is based on showing that such ideals have linear quotients. Since the regularity of a componentwise linear ideal can be determined from its generators, we obtain a formula for the regularity of powers of the vertex cover ideal in terms of the maximum degree of the graph.


\section{Powers of Edge Ideals of $(C_4, 2K_2)$-free graphs}\label{sec:edge ideals}
\subsection{Background on Graph Theory}
Throughout this paper $G$ will denote a finite simple graph. We write $G=(V,E)$ where $V$ and $E$ are respectively the sets of vertices and edges of the graph. We say a vertex $v$ is a \emph{neighbor} of $u$ if $\{u,v\}\in E$. In this case, we simply write $uv\in G$ and say $u$ and $v$ are \emph{adjacent}. The neighbor set of $v$, denoted by $N(v)$, consists of all vertices of $G$ that are adjacent to $v$. A vertex is \emph{isolated} if it has no neighbors. The \emph{complement} graph $G^c$ has the same vertices as $G$ and $uv\in G^c$ if $uv\notin G$. We say $G$ is \emph{connected} if there is a path between every pair of vertices.

A graph $H$ is called a \emph{subgraph} of $G$ if the vertex and edge sets of $H$ are contained respectively in those of $G$. A subgraph $H$ of $G$ is called an \emph{induced subgraph} if $uv\in G$ implies $uv\in H$ for all vertices $u$ and $v$ of $H$. We say $G$ is \emph{$H$-free} if $G$ has no induced subgraph isomorphic to $H$. If $A$ is a set of vertices of $G$, then $G-A$ denotes the induced subgraph which is obtained from $G$ by removing the vertices in $A$. We say $A$ is an \emph{independent} set if no two vertices of $A$ are adjacent in $G$. A \emph{complete graph} (or \emph{clique}) is a graph such that every pair of vertices are adjacent. A complete graph on $n$ vertices is denoted by $K_n$. A \emph{cycle} graph with vertices $v_1,\dots ,v_n$ and edges $v_1v_2, \dots , v_{n-1}v_n, v_nv_1$ is denoted by $C_n=(v_1v_2\dots v_n)$. A graph is \emph{chordal} if it has no induced cycle of length $4$ or more. The complement graph of $C_4$ is denoted by $2K_2$. A $2K_2$-free graph is also known as \emph{gap-free}. Notice that a gap-free graph is connected if and only if it has no isolated vertices. We will use the following theorem which describes graphs that are both $C_4$-free and $2K_2$-free. We will say such graphs are $(C_4,2K_2)$-free.

\begin{theorem}\cite[Theorem 1.1]{blazsik}\label{thm: gap free C4 free}
A graph $G=( V, E)$ is $(C_4, 2K_2)$-free if and only if there is a partition $V_1 \cup V_2 \cup V_3 =V$ with the following properties: 
\begin{itemize}
\item[(i)] $V_1$ is an independent set in $G$.
\item[(ii)] $V_2$ is the vertex set of a complete subgraph in $G$.
\item[(iii)] $V_3=\emptyset$ or $|V_3|=5$, and in the latter case $V_3$ induces a $5$-cycle in $G$.
\item[(iv)]If $V_3\neq \emptyset$, then for all $v_i\in V_i$, $i=1,2,3$, $v_1v_3 \notin E$ and $v_2 v_3\in E$ hold.
\end{itemize}
\end{theorem}

A graph is called a \emph{split} graph if its vertex set can be partitioned into an independent set and a clique. If $H$ is a split graph, then both $H$ and $H^c$ are chordal. Observe that in Theorem~\ref{thm: gap free C4 free} the induced subgraph of $G$ on $V_1\cup V_2$ is a split graph. 

\begin{remark}\label{rk: no cycle of length greater than 5}
Observe that any cycle of length at least $6$ contains $2K_2$ as an induced subgraph. Therefore, a $2K_2$-free graph has no induced $C_n$ where $n\geq 6$. Thus a $(C_4,2K_2)$-free graph has no induced $C_n$ and $C_n^c$ for $n\geq 6$.
\end{remark}

\subsection{Bounding the Regularity}
The \emph{(Castelnuovo-Mumford) regularity} of a monomial ideal $I\subset S=\Bbbk[x_1,\dots ,x_n]$ is given by
$$\reg(I)=\max\{j-i : b_{i,j}(I)\neq 0\}   $$
where $b_{i,j}(I)$ are the \emph{graded Betti numbers} of $I$. An ideal $I$ generated in degree $d$ is said to have a \emph{linear resolution} if for all $i\geq 0$ $b_{i,j}(I)=0$ for all $j\neq i+d$.

If $G$ is a graph on the vertices $x_1,\dots , x_n$, then the \emph{edge ideal} of $G$ is defined as
$$I(G)= (xy: xy \text{ is an edge of } G). $$

The method of polarization reduces the study of minimal free resolutions of monomial ideals to that of square-free monomial ideals. Therefore quadratic monomial ideals can be studied via edge ideals. The following classical result of Fr\"oberg \cite{froberg} gives a combinatorial characterization of edge ideals which have linear resolutions.

\begin{theorem}\cite[Theorem~1]{froberg}\label{thm: frobergs theorem} The minimal free resolution of $I(G)$ is linear (i.e., $\reg(I(G)=2)$ if and only if the complement graph $G^c$ is chordal.
\end{theorem}

We recall the following well-known results.

\begin{theorem}\cite[Corollary 1.6.3]{herzog hibi monomial ideals}\label{thm: polarization}
	Let $I\subset S$ be a monomial ideal and let $I^{\pol}$ be its polarization. Then $\reg(I)=\reg(I^{\pol})$.
\end{theorem}

\begin{lemma}\label{lem:regularity after adding variables} 
	If $I\subset S$ is a monomial ideal, then $\reg(I,x)\leq \reg(I)$ for any variable $x$.
\end{lemma}

\begin{lemma}\label{lem:well known}
	Let $I\subset S$ be a monomial ideal, and let $m$ be a monomial of degree $d$. Then
	$$\reg(I)\leq \max\{\reg(I:m)+d, \reg(I,m)\}. $$
	\end{lemma}

\begin{lemma}\cite[Lemma~3.1]{huneke}\label{lem: regularity bound with star x} Let $x$ be a vertex of $G$ with neighbours $y_1,\dots ,y_m$. Then
$$\displaystyle \reg(I(G))\leq \max\{\reg (I(G- \{N(x)\cup\{x\}\}))+1, \reg(I(G- \{x\})) \}. $$
Moreover, $\reg(I(G))$ is equal to one of these terms.
\end{lemma}

Given an edge ideal, the following order was defined on its powers:

\begin{definition}\cite[Discussion~4.1]{banerjee}\label{def:order on the generators of the powers of ideal} Let $I$ be an edge ideal which is minimally generated by the monomials $L_1,\dots ,L_k$. Consider the order $L_1>L_2 >\cdots > L_k$ and let $s\geq 1$. Given two minimal monomial generators $M, N$ of $I^s$  we set $M>N$ if there exists an expression $L_1^{a_1}L_2^{a_2}\dots L_k^{a_k}=M$ such that $(a_1,\dots ,a_k)>_{\lex} (b_1,\dots ,b_k)$ for every expression 
$L_1^{b_2}\dots L_k^{b_k}=N$. 
\end{definition}

The next result follows from Theorem~4.12 in \cite{banerjee}. The last part of the theorem was stated in a different form in \cite[Corollary~5.3]{banerjee}. We state the theorem in a more general form and provide the proof for the last part.

\begin{theorem}\cite{banerjee}\label{thm:ordered colon ideals}
	Let $I$ be an edge ideal which is minimally generated by the monomials $L_1,\dots ,L_{r_1}$. For each $s\geq 1$, let $L_1^{(s)}>L_2^{(s)}>\cdots >L_{r_s}^{(s)}$ be the order on the minimal monomial generators of $I^s$ induced by the order $L_1>\cdots >L_{r_1}$ as described in Definition~\ref{def:order on the generators of the powers of ideal}. Then for all $s\geq 1$ and $1\leq \ell \leq r_s-1$, $$((I^{s+1}, L_1^{(s)},\dots , L_{\ell}^{(s)}):L_{\ell+1}^{(s)})= ( (I^{s+1}:L_{\ell+1}^{(s)}), \text{ some variables}).$$ In particular, if $\reg(I)\leq 4$ and $\reg(((I^{s+1}, L_1^{(s)},\dots , L_{\ell}^{(s)}):L_{\ell+1}^{(s)}))\leq 2$ for every $0\leq \ell \leq r_s-1$ and $s\geq 1$, then $I^t$ has linear resolution for every $t\geq 2$.
\end{theorem}
\begin{proof}
	The first part of the theorem follows from \cite[Theorem~4.12]{banerjee}. To prove the last part suppose that $\reg(I)\leq 4$ and $\reg(((I^{s+1}, L_1^{(s)},\dots , L_{\ell}^{(s)}):L_{\ell+1}^{(s)}))\leq 2$ for every $0\leq \ell \leq r_s-1$ and $s\geq 1$. Let $t\geq 1$. Using Lemma~\ref{lem:well known} we get
		\begin{equation*}
		\begin{split}
		\reg(I^{t+1}) & \leq \max\{\reg(I^{t+1}:L_1^{(t)})+2t, \, \reg(I^{t+1}, L_1^{(t)})\}\\
		& \leq  \max\{2+2t, \, \reg((I^{t+1}, L_1^{(t)}):L_2^{(t)})+2t, \, \reg((I^{t+1}, L_1^{(t)}, L_2^{(t)}))\}    \\
		& =  \max\{2+2t, \, \reg((I^{t+1}, L_1^{(t)}, L_2^{(t)}))\} \\
		& \quad \vdots \\
		& \leq \max\{2+2t, \reg(I^{t+1},L_1^{(t)},\dots ,L_{r_t}^{(t)})\}\\
		& =\max\{2+2t, \reg(I^t)\}.
		\end{split}
		\end{equation*} 
		Since $\reg(I)\leq 4$, the result follows inductively.
\end{proof}
\begin{theorem}\cite[Lemmas 6.14 and 6.15]{banerjee}\label{thm: associated graph to colon ideal}
		Let $G$ be a gap-free graph with edge ideal $I=I(G)$ and let $e_1,\dots ,e_s$ be some edges of $G$ where $s\geq 1$. Then $(I^{s+1}:e_1\dots e_s)^{\pol}$ is the edge ideal of some gap-free graph $G'$. Also, if $C_n=(u_1,\dots ,u_n)$ is a cycle on $n\geq 5$ vertices such that $C_n^c$ is an induced subgraph of $G'$, then $C_n^c$ is an induced subgraph of $G$ as well.
\end{theorem}

\subsection{Regularity of Powers of Edge Ideals of $(C_4,2K_2)$-free Graphs}
We first bound the regularity of the edge ideal of a $(C_4,2K_2)$-free graph. We use Lemma~\ref{lem: regularity bound with star x} and the description of $(C_4,2K_2)$-free graphs given in Theorem~\ref{thm: gap free C4 free}.

\begin{proposition}\label{prop:reg bounded by 3}
Let $G$ be a $(C_4, 2K_2)$-free graph. Then $\reg(I(G))\leq 3$.
\end{proposition}

\begin{proof}
We may assume that $G$ has no isolated vertices as the removal of isolated vertices does not change the edge ideal. We proceed by induction on the number of vertices of $G$. If $G$ is $K_2$, then the result is clear. Let $V=V_1\cup V_2 \cup V_3$ be a partition of the vertices of $G$ as in Theorem \ref{thm: gap free C4 free}. If $V_3=\emptyset$, then $G^c$ is chordal and the result follows from Theorem~\ref{thm: frobergs theorem}. Therefore let us assume that $V_3\neq \emptyset$. Let $x$ be a vertex that belongs to $V_3$. Then both $G-\{x\}$ and $(G-\{x\})^c$ are also gap-free. By induction we have $\reg(I(G-\{x\}))\leq 3$. Also $G-\{N(x)\cup \{x\}\}$ consists of an edge and possibly some isolated vertices. Thus $\reg(I(G-\{N(x)\cup \{x\}\}))=2$ and the proof follows from Lemma~\ref{lem: regularity bound with star x}.
\end{proof}

We are now ready to prove our main result in this section.
\begin{theorem}\label{thm:main thm1}
If $G$ is a $(C_4, 2K_2)$-free graph, then $I(G)^s$ has a linear resolution for all $s\geq 2$.
\end{theorem}
\begin{proof}
Let $V=V_1\cup V_2\cup V_3$ be a partition of the vertices of $G$ as in Theorem~\ref{thm: gap free C4 free}. If $V_3=\emptyset$, then $G$ and $G^c$ are chordal and the result follows from \cite[Theorem~3.2]{herzog hibi zheng} and Theorem~\ref{thm: frobergs theorem}. Let us assume that $V_3\neq\emptyset$.  

Let $A$ be the set of edges of $G$ that contain a vertex of $V_3$. Let $B$ be the set of remaining edges of $G$. Fix a total order on the edges of $G$ such that $a>b$ for any $a\in A$ and $b\in B$. Let $s\geq 1$ be arbitrarily fixed and consider the total order $M_1>\dots >M_r$ on the minimal monomial generators of $I(G)^s$ induced by the order on the edges of $G$ as described in Definition~\ref{def:order on the generators of the powers of ideal}. By Theorem~\ref{thm:ordered colon ideals} and Proposition~\ref{prop:reg bounded by 3} it suffices to show that 
$$\reg(((I(G)^{s+1},M_1,\dots ,M_{\ell}):M_{\ell+1}))\leq 2 \ \text{ for every } 0 \leq \ell \leq r-1.$$

First, we claim that if $M$ is a minimal monomial generator of $I(G)^s$ which is divisible by a vertex of $V_3$, then $\reg((I(G)^{s+1}:M)^{\pol})=2$. To this end, let $G'$ be the gap-free graph with edge ideal $(I(G)^{s+1}:M)^{\pol}$ as in Theorem~\ref{thm: associated graph to colon ideal}. From Theorem~\ref{thm: frobergs theorem}, it suffices to show that $G'$ has no induced $C_n^c$ for $n\geq 5$. Assume for a contradiction $C_n^c$ is an induced subgraph of $G'$ for some $n\geq 5$. Then $C_n^c$ is an induced subgraph of $G$ as well by Theorem~\ref{thm: associated graph to colon ideal}. By Remark~\ref{rk: no cycle of length greater than 5} we must have $n=5$. Notice that the complement of a cycle of length $5$ is again a cycle of length $5$. Observe that the induced subgraph of $G$ on $V_3$ is the only induced cycle of $G$ of length $5$. Therefore $V(C_5^c)=V_3$. Then by \cite[Lemma~3.6]{erey} no vertex of $V_3$ divides $M$, which is a contradiction.

Let $t$ be the largest index such that $M_t$ is divisible by a vertex in $V_3$. Then each of $M_1,\dots , M_t$ is divisible by a vertex in $V_3$ and we have $\reg((I(G)^{s+1}:M_i)^{\pol})=2$ for all $1\leq i \leq t$ by the previously proved claim. Combining Lemma~\ref{lem:regularity after adding variables}, Theorem~\ref{thm: polarization} and Theorem~\ref{thm:ordered colon ideals} we get
$$\reg(((I(G)^{s+1},M_1,\dots ,M_{\ell}):M_{\ell+1}))\leq 2 \ \text{ for every } 0 \leq \ell \leq t-1.$$

Let $0\leq j \leq r-t-1$ be fixed. We will show that 
$$\reg((I(G)^{s+1}, M_1,\dots , M_{t+j}):M_{t+j+1})\leq 2.$$
We claim that 
$$(z: z\in V_3) \subseteq ((I(G)^{s+1}, M_1,\dots , M_{t+j}):M_{t+j+1}). $$
Indeed, let $M_{t+j+1}=(vw)N$ where $v\in V_2$ and $N=1$ for $s=1$ and $N$ is a minimal monomial generator of $I(G)^{s-1}$ for $s\geq 2$. Then for every $u\in V_3$ the monomial $(uv)N$ is a minimal generator of $I(G)^s$ and $(uv) N >M_{t+j+1}$. Thus $(uvN):(M_{t+j+1})=(u)$. Now, using Theorem~\ref{thm:ordered colon ideals} we get
$$(I(G)^{s+1}, M_1,\dots , M_{t+j}):M_{t+j+1})=((J^{s+1}:M_{t+j+1}), \text{ some variables})$$
where $J$ is the edge ideal of the induced subgraph of $G$ on $V_1\cup V_2$. By Theorem~\ref{thm: frobergs theorem}, the ideal $J$ has a linear resolution. From the proof of \cite[Theorem~6.16]{banerjee} and Theorem~\ref{thm: polarization} it is known that $\reg((J^{s+1}:M_{t+j+1}))=2$. Thus the proof follows from Lemma~\ref{lem:regularity after adding variables}.
\end{proof}

\begin{remark}
	In Theorem~\ref{thm:main thm1} the order on the generators of $I(G)$ is crucial. For example, let $V_1=\{a\}, V_2=\{b\}$ and $V_3=\{c, d, e, f, g\}$ with 
	$$I(G)=(ab, bc, bd, be, bf, bg, cd, de, ef, fg, gc).$$ Using Macaulay$2$~\cite{M2} we get $\reg(I(G)^2:(ab))=3$. Therefore an order on the edges which starts with $ab$ does not allow one to apply Theorem~\ref{thm:ordered colon ideals}.
\end{remark}
\section{Powers of Vertex Cover Ideals of $(C_4, 2K_2)$-free graphs}\label{sec: vertex cover ideals}

Let $G$ be a graph on the vertices $x_1,\dots ,x_n$. A set $C$ of vertices of $G$ is called a \emph{vertex cover} if every edge of $G$ contains a vertex from $C$. The vertex cover $C$ is called \emph{minimal} if no proper subset of $C$ is a vertex cover of $G$. The \emph{vertex cover ideal} of $G$, denoted by $I(G)^{\vee}$, is defined as
$$I(G)^{\vee}= (x_{i_1}\dots x_{i_k} : \{x_{i_1}, \dots , x_{i_k}\} \text{ is a minimal vertex cover of } G ). $$
In the next lemma, we describe the minimal vertex covers of $(C_4, 2K_2)$-free graphs.
\begin{lemma}\label{lem:form of mvcs} 
Let $G$ be a $(C_4, 2K_2)$-free graph with the vertex set $V=V_1\cup V_2 \cup V_3$ partitioned as in Theorem~\ref{thm: gap free C4 free} and let $V_3\neq \emptyset$. Then any $A \subseteq V$ is a minimal vertex cover of $G$ if and only if $A$ has one of the following forms:
\begin{itemize}
\item[(i)] $A=V_2\cup \{a,b,c\}$ where $a,b,c\in V_3$, $ab\in G$, $ac\notin G$ and $bc\notin G$,
\item[(ii)] $A=N(a)$ for some $a\in V_2$.
\end{itemize}  
\end{lemma}

\begin{proof}
One can show that if $A$ satisfies one of the conditions above, then it is a minimal vertex cover. Conversely, suppose that $A$ is a minimal vertex cover. Let $C_5$ be the induced subgraph on $V_3$. We consider cases.

\emph{Case 1:} Suppose that there exists a vertex $u\in V_3$ such that $u\notin A$. Since $u$ is adjacent to every vertex in $V_2$ we must have $V_2\subseteq A$. This implies $A\cap V_1=\emptyset$ and $A\setminus V_2$ is a minimal vertex cover of $C_5$. Then $A$ must be of the form given in $(i)$ since any minimal vertex cover of $C_5$ consists of three vertices $a, b, c$ such that $ab\in C_5, ac\notin C_5$ and $bc \notin C_5$.

\emph{Case 2:} Suppose that $V_3\subseteq A$. Since any minimal vertex cover of $C_5$ contains $3$ vertices, there exists $a\in V_2$ such that $a\notin A$. Observe that any minimal vertex cover of $K_n$ has $n-1$ vertices. Therefore $V_2\setminus \{a\}\subseteq A$. Since $A$ covers $G$, we have $N(a)\cap V_1\subseteq A$. By minimality of $A$ no vertex of $V_1$ except the neighbors of $A$ can belong to $A$. Thus $A=N(a)$ as in $(ii)$.
\end{proof}

\begin{definition}[Linear quotients]
A monomial ideal $I\subset \Bbbk[x_1,\dots ,x_n]$ is said to have \emph{linear quotients} if there is an ordering $m_1,\dots ,m_q$ on the minimal monomial generators of $I$ such that for every $i=2,\dots ,q$ the ideal $(m_1,\dots ,m_{i-1}):m_i$ is generated by a subset of $\{x_1,\dots ,x_n\}$.
\end{definition}

A connected graph is called a \emph{cactus} if each edge belongs to at most one cycle. A \emph{generalized star graph} based on $G_{m,n_1,\dots ,n_k}$ is a special type of chordal graph, see \cite[Definition~1.3]{mohammadi powers of chordal}. Mohammadi \cite{mohammadi cactus, mohammadi powers of chordal} proved the following result regarding the powers of vertex cover ideals of these graphs.

\begin{theorem}\cite[Theorem~3.3]{mohammadi cactus}\cite[Theorem~1.5]{mohammadi powers of chordal}\label{thm:mohammadi}
If $G$ is either a Cohen-Macaulay cactus graph or a generalized star graph based on $G_{m,n_1,\dots ,n_k}$, then all powers of the vertex cover ideal of $G$ are weakly polymatroidal. In particular, they have linear quotients.
\end{theorem}

\begin{lemma}\label{lem:unique expression}
	Let $f_1,\dots ,f_5$ be the minimal vertex covers of $C_5$ and let $s\geq 1$. Then every minimal monomial generator $M$ of $(I(C_5)^{\vee})^s$ has a unique expression of the form $M=f_1^{\alpha_1}\dots f_5^{\alpha_5}$.
\end{lemma}
\begin{proof}
Let $I(C_5)=(u_1u_2, u_2u_3, u_3u_4, u_4u_5, u_5u_1)$ and let $f_1=u_1u_2u_4, f_2=u_4u_5u_2, f_3=u_2u_3u_5, f_4=u_3u_4u_1, f_5=u_5u_1u_3$. Suppose $M=f_1^{\beta_1}\dots f_5^{\beta_5}$ is another expression. Since each $f_i$ has the same degree, $\sum_{i=1}^{5}\alpha_i=\sum_{i=1}^{5}\beta_i$. The exponent of $u_2$ in $M$ is $\sum_{i=1}^{3}\alpha_i=\sum_{i=1}^{3}\beta_i$. Therefore $\alpha_4+\alpha_5=\beta_4+\beta_5$. The exponent of $u_1$ in $M$ is $\alpha_1+\alpha_4+\alpha_5=\beta_1+\beta_4+\beta_5$. Thus $\alpha_1=\beta_1$. From the symmetry of the graph it follows that $\alpha_i=\beta_i$ for each $i=1,\dots ,5$.
\end{proof}

\begin{lemma}\label{lem:every s-fold product is minimal}
	Let $G$ be a $(C_4, 2K_2)$-free graph with the vertex set $V=V_1\cup V_2\cup V_3$ partitioned as in Theorem~\ref{thm: gap free C4 free} and let $V_3\neq \emptyset$. Let $M_1,\dots ,M_q$ be the minimal monomial generators of $I(G)^{\vee}$. Then for every $s\geq 1$ and non-negative integers $m_1,\dots, m_q$ with $\sum_{i=1}^qm_i=s$, the monomial $M_1^{m_1}\dots M_q^{m_q}$ is a minimal generator of $(I(G)^{\vee})^s$.
\end{lemma}
\begin{proof}
	If $V_2=\emptyset$, then the result is clear since $(I(G)^{\vee})^s$ is generated in degree $3s$. So, let us assume that $V_2=\{z_1,\dots ,z_k\}$ for some $k\geq 1$. Let $I(C_5)^{\vee}=(f_1,\dots ,f_5)$. Using Lemma~\ref{lem:form of mvcs} let 
	$$G=\prod_{i=1}^{5}(V_2f_i)^{\alpha_i}\prod_{j=1}^{k}N(z_j)^{\beta_j}$$
be a generator of $(I(G)^{\vee})^s$	for some $\alpha_i,\beta_j\geq 0$ with $\sum_{i=1}^{5}\alpha_i+\sum_{j=1}^{k}\beta_j=s$. We claim that $G$ is minimal. Let
	$$H=\prod_{i=1}^{5}(V_2f_i)^{\gamma_i}\prod_{j=1}^{k}N(z_j)^{\kappa_j} $$
	be a minimal generator of $(I(G)^{\vee})^s$ such that $H$ divides $G$. Observe that for every $j=1,\dots ,k$ the exponents of $z_j$ in $G$ and $H$ are respectively $s-\beta_j$ and $s-\kappa_j$. Since $H$ divides $G$ we get $\kappa_j\geq \beta_j$ for all $j=1,\dots,k$. Assume for a contradiction that $\kappa_{j_0}>\beta_{j_0}$ for some $j_0$. Then we get $\sum_{i=1}^{5}\alpha_i>\sum_{i=1}^{5}\gamma_i$. Let $u=\gcd(H, V_3^s)$ and $v=\gcd(G,V_3^s)$ so that $u$ divides $v$. Observe that $u$ has degree
	$$3\sum_{i=1}^{5}\gamma_i +5\sum_{j=1}^{k}\kappa_j =3\sum_{i=1}^{5}\gamma_i+5(s-\sum_{i=1}^{5}\gamma_i)=5s-2\sum_{i=1}^{5}\gamma_i.$$
	Similarly, $v$ has degree $5s-2\sum_{i=1}^{5}\alpha_i$. Thus degree of $u$ is greater than degree of $v$, which is a contradiction. Now, we have $\kappa_j=\beta_j$ for all $j=1,\dots ,k$. This implies $\sum_{i=1}^{5}\alpha_i=\sum_{i=1}^{5}\gamma_i$. Then $\prod_{i=1}^{5}f_i^{\alpha_i}$ and $\prod_{i=1}^{5}f_i^{\gamma_i}$ are minimal generators of $(I(C_5)^{\vee})^A$ where $A=\sum_{i=1}^{5}\alpha_i$. Since $H$ divides $G$ we get $\prod_{i=1}^{5}f_i^{\alpha_i}=\prod_{i=1}^{5}f_i^{\gamma_i}$. Thus $G=H$ and the minimality of $G$ is established.
\end{proof}
\begin{lemma}\label{lem: split is generalized star}
	If $G$ is a connected split graph with at least one edge, then $G$ is a generalized star graph based on $G_{m,n_1}$ for some $m,n_1$.
\end{lemma}
\begin{proof}
	Suppose that $G$ is a split graph with the vertex set $V=V_1\cup V_2$ where $V_1$ is independent and $V_2$ is a clique. We may assume that $V_2$ is a maximal clique of $G$. If $V_1=\emptyset$, then $G=G_{1,|V_2|-1}$. Let $u_1\in V_1$ and let $N(u_1)=\{y_1,\dots ,y_m\}$. Since $V_2$ is a maximal clique we have $V_2\setminus N(u_1)\neq \emptyset$. Let $V_2\setminus N(u_1)=\{x_{1,1}, \dots , x_{1,n_1}\}$ and $V_1\setminus \{u_1\}=\{u_{1,1},\dots , u_{1,m_1}\}$. Now we have
	\begin{itemize}
	\item[(i)]  $N(u_1)\subseteq \{y_1,\dots ,y_m\}$
	\item[(ii)]  $\cup_{j=1}^{m_1} N(u_{1,j})\subseteq \{x_{1,1},\dots ,x_{1,n_1}\}\cup \{y_1,\dots ,y_m\}$
	\item[(iii)]  $N(y_{\ell})\cap \{u_1\}\neq \emptyset$ for all $\ell=1,\dots ,m$.
	\end{itemize}
	which shows that $G$ is a generalized star graph based on $G_{m,n_1}$. 
\end{proof}
We now prove our main result in this section.
\begin{theorem}\label{thm: linear quotients of powers}
	Let $G$ be a $(C_4,2K_2)$-free graph with at least one edge. Then $(I(G)^{\vee})^s$ has linear quotients for every $s\geq 1$.
\end{theorem}
\begin{proof}
	We may assume that $G$ has no isolated vertices. Suppose that the vertex set $V=V_1\cup V_2 \cup V_3$ of $G$ is partitioned as in ~Theorem~\ref{thm: gap free C4 free}. If $V_2=\emptyset$, then $I(G)=I(C_5)$. Since $C_5$ is a connected cactus graph and it is Cohen-Macaulay by \cite[Proposition~4.1]{francisco van tuyl} it follows from Theorem~\ref{thm:mohammadi} that every power of $I(C_5)^{\vee}$ has linear quotients. If $V_3=\emptyset$, then the proof follows from Lemma~\ref{lem: split is generalized star} and Theorem~\ref{thm:mohammadi}. Therefore let us assume that both $V_2$ and $V_3$ are non-empty. Let $V_2=\{z_1,\dots, z_k\}$ and let $I(C_5)^{\vee}=(f_1, \dots , f_5)$. We claim that every minimal generator $M$ of $(I(G)^{\vee})^s$ has a unique expression of the form $M=\prod_{i=1}^{5}(V_2f_i)^{\alpha_i}\prod_{j=1}^{k}N(z_j)^{\beta_j}$. Note that $M$ has such expression by Lemma~\ref{lem:form of mvcs}. Suppose that $M=\prod_{i=1}^{5}(V_2f_i)^{\gamma_i}\prod_{j=1}^{k}N(z_j)^{\kappa_j}$. For each $j\in\{1,\dots ,k\}$ the exponent of $z_j$ in $M$ is $s-\beta_j=s-\kappa_j$ and thus $\beta_j=\kappa_j$. Then $\sum_{i=1}^{5}\alpha_i=\sum_{i=1}^{5}\gamma_i$ and from Lemma~\ref{lem:unique expression} it follows that $\alpha_i=\gamma_i$ for each $i\in\{1,\dots, 5\}$.
	
	Let $L_1^p<L_2^p < \dots < L_{r_p}^p$ be a linear quotients order on the minimal monomial generators of $(I(C_5)^{\vee})^p$ for each $p$. Consider a total order $M_1,\dots ,M_q$ on the minimal monomial generators of $(I(G)^{\vee})^s$ such that for any
	\begin{equation}\label{eq:generators product form}
	M_t=\prod_{i=1}^{5}(V_2f_i)^{\alpha_i}\prod_{j=1}^{k}N(z_j)^{\beta_j} \quad \text{and} \quad 	M_{\ell}=\prod_{i=1}^{5}(V_2f_i)^{\gamma_i}\prod_{j=1}^{k}N(z_j)^{\kappa_j} 
	\end{equation}
	$M_t$ precedes $M_{\ell}$ in the order (i.e., $t<\ell$) if one of the following holds:
	\begin{itemize}
		\item[(i)] $\displaystyle\sum_{i=1}^{5}\alpha_i > \sum_{i=1}^{5}\gamma_i $
		\item[(ii)] $(\alpha_1,\dots,\alpha_5)=(\gamma_1,\dots ,\gamma_5)$ and $(\beta_1,\dots , \beta_k)>_{\lex}(\kappa_1,\dots ,\kappa_k)$
		\item[(iii)] $\displaystyle A:=\sum_{i=1}^{5}\alpha_i = \sum_{i=1}^{5}\gamma_i $ and $\displaystyle \prod_{i=1}^{5}f_i^{\alpha_i} < \displaystyle \prod_{i=1}^{5}f_i^{\gamma_i}$ in the linear quotients order of $(I(C_5)^{\vee})^A$.
	\end{itemize}
	Let $2\leq \ell \leq q$ be as in Eq~\eqref{eq:generators product form}. We will show that $(M_1,\dots ,M_{\ell-1}):M_{\ell}$ is generated by variables. If $\kappa_j=0$ for all $j=1,\dots ,k$, then for all $t<\ell$ with $M_t$ as in Eq.~\eqref{eq:generators product form} the condition $(iii)$ is satisfied. In this case, the colon ideal
	$$(M_1,\dots ,M_{\ell-1}):M_{\ell} = (L_1^s, \dots , L_{\ell-1}^s):L_{\ell}^s  $$
	is generated by some variables. So, let us assume that $\kappa_j\neq 0$ for at least one $j$. We claim that $z_j$ is a generator of $(M_1,\dots ,M_{\ell-1}):M_{\ell}$ for every $\kappa_j\neq 0$. Indeed, if $\kappa_j\neq 0$, then by Lemma~\ref{lem:every s-fold product is minimal} $M_r=(M_{\ell}V_2f_5)/N(z_j)$ is a minimal monomial generator for some $r<\ell$. Observe that $M_r/\gcd(M_r,M_{\ell})=z_j$ which shows that $z_j$ is a generator of the colon ideal.
	
	 Let $t<\ell$ and let $M_t$ be as in Eq.~\eqref{eq:generators product form}. Observe that if $\beta_j<\kappa_j$ for some $j$, then $z_j$ divides $M_t/\gcd(M_t,M_{\ell})$. If $\beta_j\geq \kappa_j$ for all $j$, then we have $(\beta_1,\dots ,\beta_k)=(\kappa_1,\dots ,\kappa_k)$ and the condition $(iii)$ holds. Thus using Lemma~\ref{lem:every s-fold product is minimal} we get
	$$(M_1,\dots ,M_{\ell-1}):M_{\ell}=(z_j : \kappa_j\neq 0) + (\text{possibly some vertices of} \ V_3)  $$
	as desired.
		\end{proof}
	
If $v$ is a vertex of a graph $G$, then the \emph{degree} of $v$ is the number of neighbors of $v$. The \emph{maximum degree} of $G$, denoted by $\Delta(G)$, is the maximum degree of its vertices.
\begin{cor}
	If $G$ is a $(C_4,2K_2)$-free graph, then for each $s\geq 1$
	\[\reg((I(G)^{\vee})^s) = \begin{cases} 
	3s & \text{if }  G \text{ is a } C_5 \text { with isolated vertices} \\
	s\Delta(G) & \text{otherwise.} 
	\end{cases}
	\]	
\end{cor}
	\begin{proof}
	It is known \cite[Theorem~8.2.15]{herzog hibi monomial ideals} that if an ideal has linear quotients, then it is componentwise linear. Therefore by \cite[Corollary~8.2.14]{herzog hibi monomial ideals} for each $s\geq 1$, $\reg((I(G)^{\vee})^s)$ is equal to the highest degree of a generator in a minimal set of generators of $(I(G)^{\vee})^s$.  Let $V=V_1\cup V_2\cup V_3$ be partitioned as in Theorem~\ref{thm: gap free C4 free}. Let us assume that $V_2\neq \emptyset$ since otherwise the result is already known. 
	
\emph{Case 1:} Suppose that $V_3\neq \emptyset$. Then by Lemma~\ref{lem:form of mvcs}, for any minimal vertex cover $A$ of $G$ and for all $v_i\in V_i$ we have
$$|N(v_1)|< |N(v_3)| < |A| \leq |N(v_2)| $$ 
which shows that maximum cardinality of a minimal vertex cover of $G$ is $\Delta(G)$.

\emph{Case 2:} Suppose that $V_3=\emptyset$. First note that $N(a)$ is a minimal vertex cover of $G$ for every $a\in V_2$. Also, notice that $V_2$ is a minimal vertex cover of $G$ if and only if $N(a)\cap V_1\neq \emptyset$ for every $a\in V_2$. If $A$ is a minimal vertex cover of $G$, then either $A=V_2$ or $A=N(a)$ for some $a\in V_2$. Therefore the maximum cardinality of a minimal vertex cover of $G$ is $\Delta(G)$, which completes the proof.
	\end{proof}

\end{document}